\documentclass[12pt]{amsart}
\usepackage{amsmath,amsfonts,amssymb,amsthm,amstext,pgf,graphicx,hyperref,verbatim,lmodern,textcomp,color,young,tikz}
\usetikzlibrary{decorations}
\usepackage[mathscr]{euscript}
\usetikzlibrary{decorations.markings}
\usetikzlibrary{arrows}
\usepackage{float}
\theoremstyle{plain}
\newtheorem{theorem}{Theorem}[section]
\newtheorem{lemma}[theorem]{Lemma}

\newtheorem{corollary}{Corollary}

\theoremstyle{definition}
\newtheorem{defn}{Definition}[section]

\title{}
\begin{document}
\title[On the connectivity of enhanced power graph of 
finite group]{On the connectivity of enhanced power graph of 
finite group}
\author[Sudip Bera]{Sudip Bera}
\author[Hiranya Kishore Dey]{Hiranya Kishore Dey}
\author[Sajal Kumar Mukherjee]{Sajal Kumar Mukherjee}
\address[Sudip Bera]{Department of Mathematics, Indian Institute of Science, Bangalore 560 012}
\email{sudipbera@iisc.ac.in}
\address[Hiranya Kishore Dey]{Department of Mathematics, Indian Institute of Technology, Bombay, India.}
\email{hkdey@math.iitb.ac.in}
\address[Sajal Kumar Mukherjee]{Department of Mathematics, Indian Institute of Science, Bangalore 560 012}
\email{sajalm@iisc.ac.in}
\maketitle
\begin{abstract}
This paper deals with the vertex connectivity of enhanced power graph of finite group. We classify all abelian groups $G$ such that vertex connectivity of enhanced power graph of $G$ is $1.$ We derive an upper bound of vertex connectivity for the enhanced power graph of any general abelian group $G.$ Also we completely characterize all abelian group $G,$ such that the proper enhanced power graph is connected. Moreover, we study some special class of non-abelian group $G$ such that the proper enhanced power graph is connected and we find their vertex connectivity. 
\end{abstract}
\section{Introduction}
\label{sec:intro}
The exploration of graphs associated with algebraic structures is important, as graphs like these enrich
both algebra and graph theory. Besides, they have important applications (see, for example, 
\cite{surveypwrgraphkac1, cayleygraphsckry}) and are related to automata theory \cite{automatatheory}. During the last two decades, investigation of the interplay
between the properties of an algebraic structure $S$ and the graph-theoretic properties of $\Gamma(S),$ a graph
associated with $S,$ has been an exciting topic of research. Different types of graphs, specifically zero-divisor graph of a ring \cite{anderson}, semiring \cite{atani1}, semigroup \cite{deMeyer}, poset \cite{joshizerodivgraphofideal}, power graph of semigroup \cite{undpwrgraphofsemgmainsgc1, directedgrphcompropofsemgrpkq3}, group \cite{combinatorialpropertyandpowergraphsofgroupskq1}, normal subgroup based power graph of group \cite{normalsubgrpbasedpwrbb2}, intersection power graph of group \cite{intersectionpwegraphb3} etc. have been introduced to study algebraic structures using graph theory. Chakrabarty et. al in \cite{undpwrgraphofsemgmainsgc1} introduced the undirected power graph of a semigroup in the following way.
\begin{defn}[\cite{undpwrgraphofsemgmainsgc1}]\label{defn: powr graph}
Let $S$ be a semigroup, then the \emph{power graph} $\mathcal{P}(S)$ of $S,$ is a simple graph, whose vertex set is $S$ and two distinct vertices $u$ and $v$ are edge connected if and only if either $u^m=v$ or $v^n=u,$ where $m, n\in \mathbb{N}.$    
\end{defn}	
Another well-studied graph, called \emph{commuting graph} associated with a group $G$ is studied in \cite{braurflower} as a part of the classification of finite simple groups.  For more information about the commuting graph, see \cite{ijraeljofmathematics, onboundingdiamcommuting}.
\begin{defn}[\cite{braurflower}]\label{defn: commuting graph}
Let $G$ be a group, then the \emph{commuting graph} of $G,$ denoted by $\mathcal{C}(G),$ is the simple graph whose vertex set is non-central elements of $G$ and two distinct vertices $u$ and $v$ are adjacent if and only if $uv=vu.$ 	
\end{defn}
\begin{defn}[\cite{firstenhcedpwrstrctreaacns1}]\label{defn: enhcdpowr graph}
Given a group $G,$ the \emph{enhanced power graph} of $G,$
denoted by $\mathcal{G}_e(G),$ is the graph with vertex set $G,$ in which $u$ and $v$ are joined if and only
if there exists an element $w \in G$ such that both $u$ and $v$ are powers of $w.$
\end{defn}
The authors in \cite{firstenhcedpwrstrctreaacns1} measure how close the power graph is to the commuting graph by using the enhanced power graph. In fact, the enhanced power graph contains the power graph and is a subgraph of the commuting graph. They characterized the finite groups such that, for an arbitrary pair of these three graphs for which this pair of graphs are equal.
Besides, in \cite{Zahirovienhnacedpwrgraph}, the researchers proved that finite groups with isomorphic enhanced power graphs have isomorphic directed power graphs. They showed that any isomorphism between the undirected power graph of finite groups is an isomorphism between enhanced power graphs of these group.
Ma and She  in \cite{ma-she} derived the metric dimension of enhanced power graphs of finite groups  where as Hamzeh et.al in \cite{Hamzeh-ashrafi} derived the automorphism groups 
of enhanced power graphs of finite groups.
Recently Panda et.al in \cite{panda-dalal-kumar} have studied independence number, vertex covering number and some other graph invariants of enhanced power graphs.
\subsection{Basic definitions, Notations and Main Results}\label{subsection:main results}
For the convenience of the reader and also for later use, we recall some basic
 definitions and notations about graphs. 
Let $\Gamma=(V, E)$ be a graph where $V$ is the set of vertices and $E$ is the set of edges. Two elements $u$ and $v$ are said to be adjacent if $(u, v) \in E.$ The standard distance between two vertices $u$ and $v$ in a connected graph $\Gamma$ is denoted by $d(u, v).$ Clearly, if $u$ and $v$ are adjacent, then $d(u, v)=1.$ For a graph $\Gamma,$ its \emph{diameter} is defined as $diam(\Gamma)= \max_{u, v \in V} d(u, v).$ That is, the diameter of graph is the largest possible distance between pair of vertices of a graph. A \emph{path} of length $l$ between two vertices $v_0$ and $v_k$ is an alternating
sequence of vertices and edges $v_0, e_0, v_1, e_1, v_2, \cdots , v_{k-1}, e_{k-1}, v_k$, where the $v_i'$s are distinct
(except possibly the first and last vertices) and $e_i$ is the edge $(v_i, v_{i+1}).$ A graph $\Gamma$ is said to be \emph{connected} if for any pair of vertices $u$ and $v,$ there exists a path between $u$ and $v.$ $\Gamma$ is said to be \emph{complete} if any two distinct vertices are adjacent. A vertex of a graph $\Gamma=(V, E)$ is called a 
\emph{dominating vertex} if it is adjacent to every other
vertex. For a graph $\Gamma,$ let $Dom(\Gamma)$ denote the set of all dominating vertices in $\Gamma.$ The \emph{vertex connectivity} of a graph $\Gamma,$ denoted by $\kappa{(\Gamma)}$ is
the minimum number of vertices which need to be removed from the vertex set $\Gamma$ so that the
induced subgraph of $\Gamma$ on the remaining vertices is disconnected. The complete graph with $n$ vertices has connectivity $n-1.$ For more on graph theory we refer \cite{graphthrybondymurti, algbraphgodsil, graphthrywest}. The enhanced power graph is called \emph{dominatable} if it has a dominating vertex other than identity.

Throughout this paper we consider $G$ as a finite group.  $|G|$ denotes the cardinality of the set $G.$ For a prime $p,$ a group $G$ is said to be a $p$-group if $|G|=p^{r}, r\in \mathbb{N}.$ For any element $g \in G, \text{o}(g)$ denotes the order of the element $g \in G.$ Let $G$ be a group and $a\in G,$ then $\mathcal{G}en(a)$ is the set of all generators of the cyclic group $\langle a\rangle.$ Let $m$ and $n$ be any two positive integers, then the greatest common divisor of $m$ and $n$ is denoted by $\text{gcd}(m, n).$ The Euler's phi function $\phi(n)$ is the number of integers $k$ in the range $1 \leq k \leq n$ for which the  $\text{gcd}(n, k)$ is equal to $1.$ The set $\{1, 2, \cdots, n\}$ is denoted by $[n].$
  
In this paper, our focus is on the vertex connectivity of enhanced power graphs of finite abelian groups. If $G$ is a non-cyclic non-generalized quaternion $p$-group, then we determine the exact value of the vertex connectivity of $\mathcal{G}_e(G).$
\begin{theorem}\label{vc=1 for abln and non abln group}
Let $G$ be a finite $p$-group such that $G$ is neither cyclic nor generalized quaternion group. Then $\kappa(\mathcal{G}_e(G))=1.$	
\end{theorem}
Our next result classifies all non-cyclic abelian groups $G$ such that 
$\kappa(\mathcal{G}_e(G))=1.$ 
\begin{theorem}
	\label{ enhcd vc 1 iff g is p-group}
	Let $G$ be a finite non-cyclic abelian group. Then $\kappa(\mathcal{G}_e(G))$ is $1$ if and only if $G$ is a $p$-group.  
\end{theorem}
The authors in \cite[Question 40]{firstenhcedpwrstrctreaacns1} asked about the connectivity of power graphs when all the dominating vertices are removed. Recently, Cameron and Jafari in \cite{HeidarJafari} answered this question for power graphs. In this paper, we investigate the same question for enhanced power graphs. To seek the answer of this question, define the following graph:
\begin{defn}\label{defn: proper enhacd pwr graph}
Given a group $G,$ the \emph{proper enhanced power graph} of $G,$ denoted by $\mathcal{G}^{**}_e(G),$ is the graph obtained by deleting all the dominating vertices from the enhanced power graph $\mathcal{G}_e(G).$ Moreover, by $\mathcal{G}^{*}_e(G)$ we denote the graph obtained by deleting only the identity element of $G$ and this is called \emph{deleted enhanced power graph} of $G.$ Note that if there is no such dominating vertex other than identity, then $\mathcal{G}_{e}^{*}(G)=\mathcal{G}_{e}^{**}(G).$
\end{defn}
In \cite{enhancedpwrgrapbb3}, Bera et.al characterized all abelian groups $G,$ such that $|\text{Dom}(\mathcal{G}_e(G))|>1.$ In fact, they proved the folllowing:  
\begin{theorem}[\cite{enhancedpwrgrapbb3}]\label{thm:bb3 ehced, dom iff cylic syllow}
Let $G$ be a finite abelian group. Then $\mathcal{G}_e(G)$
is dominatable if and only if $G$ has a cyclic Sylow subgroup. 
\end{theorem}
So, from Theorem \ref{thm:bb3 ehced, dom iff cylic syllow}, 
$|\text{Dom}(\mathcal{G}_e(G))|>1$ if and only if $G=G_1\times \mathbb{Z}_n,$ where $\text{gcd}(|G_1|, n)=1$ and $G_1$ has no cyclic sylow subgroup. Then one natural question is that which are the dominatable vertices of $\mathcal{G}_e(G).$ The next theorem gives the complete list of the dominating vertices of $\mathcal{G}_e(G).$   
\begin{theorem}\label{thm: all dom of G times Zn}
Let $G_1$ be a non-cyclic abelian group such that $G_1$ has no cyclic sylow subgroup. If $n\in \mathbb{N}, \text{ and gcd}(|G_1|, n)=1$, then $\text{Dom}(\mathcal{G}_e({G_1\times \mathbb{Z}_n}))=\{(e, x), \text{ where } x \text{ is any element of } \mathbb{Z}_n  \text{ and } e \text{ is the identity of } G_1\}.$ 
\end{theorem}
Theorem \ref{ enhcd vc 1 iff g is p-group} completely characterizes the connectivity of $\mathcal{G}_{e}^{*}(G)$ for any finite abelian $p$-group $G.$ Now if $G$ is a non-cyclic abelian non $p$-group such that $G$ has no  cyclic sylow subgroup, then by Theorem \ref{thm:bb3 ehced, dom iff cylic syllow}, $G$ has no dominating vertex other than the identity. So, in this case we care about the connectivity of $\mathcal{G}_e^{**}(G)=\mathcal{G}_e^*(G)$ and by  Theorem \ref{ enhcd vc 1 iff g is p-group}, $\mathcal{G}_e^*(G)$ is connected.	
Therefore when the graph $\mathcal{G}_e^*(G)$ has a dominating vertex other than ientity, the connectivity of $\mathcal{G}_{e}(G)$ is a more interesting question. 
In this paper,  we characterize for which finite abelian groups, the proper enhanced power graphs $\mathcal{G}_e^{**}(G)$ are connected and for which
they are not. Our contributions on this paper in this theme is the following:
\begin{theorem}
\label{thm:**disconnected}
Let $G$ be a non-cyclic abelian non $p$-group such that $G \cong G_1 \times \mathbb{Z}_n, \text{gcd}(|G_1|, n) =1$ and $G_1$ has no cylcic sylow subgroup. Then $\mathcal{G}^{**}_e(G)$ is disconnected if and only if $G_1$ is a $p$-group. 
\end{theorem}
Therefore, from Theorem \ref{thm:**disconnected}, when 
$G_1$ is not a $p$-group,
$\mathcal{G}^{**}_e(G)$ remains connected. Thus, the number of additional vertices required to make it disconnected is an interesting question. On this theme, our next result is the following:   
\begin{theorem}\label{vc of enhced pwr raph for grnrral abelian grp}	
	Let $G$ be a non-cyclic abelian group such that \[G\cong\mathbb{Z}_{p^{t_{11}}_1}\times \mathbb{Z}_{p^{t_{12}}_1}\times\mathbb{Z}_{p^{t_{13}}_1}\times\cdots\times\mathbb{Z}_{p^{t_{1k_1}}_1}\times \mathbb{Z}_{p^{t_{21}}_2}\times\mathbb{Z}_{p^{t_{22}}_2}\times\cdots\times\mathbb{Z}_{p^{t_{2k_2}}_2}\times\cdots\times
	\mathbb{Z}_{p^{t_{r1}}_r}\times\mathbb{Z}_{p^{t_{r2}}_r}\times\cdots\times\mathbb{Z}_{p^{t_{rk_r}}_r},\] where $k_i\geq 1 $ and $1\leq t_{i1}\leq t_{i2}\leq\cdots\leq t_{ik_i} $, for all $i\in [r]. $ Then \[\kappa(\mathcal{G}_e(G))\leq p_1^{t_{11}}p_2^{t_{21}}\cdots p_r^{t_{r1}}-\phi(p_1^{t_{11}}p_2^{t_{21}}\cdots p_r^{t_{r1}}).\]
\end{theorem}
When $G_1$ is a $p$-group, the following result gives the exact value of the vertex connectivity of $\mathcal{G}_e(G).$ 
\begin{theorem}\label{vc of G1 cross cyclic}
Let $G$ be a non-cyclic abelian non-$p$-group such that $G \cong G_1 \times \mathbb{Z}_n, \text{gcd}(|G_1|, n) =1$ and $G_1$ is a $p$-group with no cyclic sylow subgroup. Then $\kappa(\mathcal{G}_e(G))=n.$	
\end{theorem}
Throughout this paper, the group operation of any abelian group is taken to be additive.
\section{Preliminaries}
\label{sec:prelim}
We first recall some earlier known  results on enhanced power graphs which we will need throughout the paper.
In \cite{enhancedpwrgrapbb3}, Bera et.al studied about the completeness, dominatability and many other properties of enhanced power graph of finite group. In fact, they proved the following:
\begin{lemma}[Theorem 2.4, \cite{enhancedpwrgrapbb3}]\label{enhd coplte iff cyclic}
The enhanced power graph $\mathcal{G}_e(G)$ of the group $G$ is complete if and only if $G$ is cyclic.
\end{lemma}
\begin{lemma}[Theorem 3.3, \cite{enhancedpwrgrapbb3}]\label{dom of gen Q_2^n in enhced}
Let $G$ be a non-abelian $2$-group. Then the enhanced power graph $\mathcal{G}_e(G)$ is dominatable  if and only if $G$ is generalized quarternion group.
\end{lemma}
\begin{lemma}[Theorem 3.1, \cite{enhancedpwrgrapbb3}]\label{dom if gcd =1}
Let $G$ be a finite group and $n \in \mathbb{N}$. If $\text{gcd}(|G|, n)=1$, then the enhanced power graph $\mathcal{G}_e(G \times \mathbb{Z}_n)$ 
is dominatable.
\end{lemma}
\begin{lemma}[Lemma 2.1, \cite{enhancedpwrgrapbb3}]\label{same ord elmnt but disnt cyclc sbg , not adjacnt i enhacd pwr grp}
Let $a, b\in G$ with $\text{o}(a)=\text{o}(b)$ and $\langle a\rangle\neq \langle b\rangle$. Then $x$ is not adjacent with $y$  for every $x\in \mathcal{G}en(a)$ and $y\in \mathcal{G}en(b).$
\end{lemma}
%
%
%
%
%
We next prove some important lemmas which are used to prove our main theorems.
\begin{lemma}
\label{lemma:any_odtwo_commuting_elements_of_prime_order_adjacent}
Let $G$ be a finite group and $x, y \in G \setminus
\{e\}$ be such that $\text{gcd}(\text{o}(x), \text{o}(y)) = 1 $ and $ xy = yx.$ Then, $x \sim y$ in $\mathcal{G}_e^*(G)$. 
\end{lemma}
\begin{proof}
Let $\text{o}(a)=m$ and $\text{o}(b)=n.$ Now $\text{gcd}(m, n)=1$ implies that $n^{\phi(m)}=mk+1, k\in\mathbb{N}, (\text{ by Euler's formula }).$ Again, $ab=ba$ implies that $(ab)^{n^{\phi(m)}}=a^{n^{\phi(m)}}b^{n^{\phi(m)}}=a^{n^{\phi(m)}}=a^{mk+1}=a.$ As a result, $a\in \langle ab \rangle.$ Similarly we can prove that $b\in \langle ab\rangle.$ Consequently, $a\sim b$ in $\mathcal{G}^{*}(G).$ 	
\end{proof}	
\begin{lemma}\label{lema: p and p^i orderd path connd abelong< b>}
Let $G$ be a $p$-group. Let $a, b$ be two elements of $G$ of order $p, p^i (i\geq 1)$ respectively. If there is a path between $a$ and $b$ in $\mathcal{G}_e^*(G),$ then $\langle a\rangle\subset \langle b\rangle.$	In particular, if both a and b have order p, then, $\langle a \rangle = \langle b \rangle.$
\end{lemma}
\begin{proof}
Let $a=a_1\sim a_2\sim \cdots\sim a_m=b$ be  a path between $a$ and $b.$ Now $a=a_1\sim a_2$ implies that there exists $x\in G$ such that $a_1, a_2\in \langle x\rangle.$ As a result, $a_1\in \langle a_2\rangle$ (since a cyclic group has a unique subgroup corresponding to each divisor of the order of the cyclic group). Now, $a_2\sim a_3$ and $G$ is a $p$-group, then either $a_2\in \langle a_3\rangle$ or $a_3\in \langle a_2\rangle.$ Clearly for  both of the cases $a_1\in \langle a_3\rangle.$ Continuing this process we can conclude that $a\in \langle b\rangle.$	
\end{proof}
\begin{lemma}
\label{lem:structurelemmauniqueprimeordersubgroup}
	Let $G$ be any non-cyclic group. For any dominating vertex $v(\neq e)$ of $G$ there exists a prime $p$ dividing $\text{o}(v)$ such that $G$ has a unique subgroup of order $p.$
\end{lemma}
\begin{proof}
	Let $v \neq e$ be a dominating vertex and $\text{o}(v) = m$. Let $p$ be a prime divisor of $m$ and $m=rp.$  We
	claim that $H=\langle v^r \rangle $ is the unique subgroup of order $p$  in $G.$ Consider $x \in G$ such that $\text{o}(x) = p.$
	Since $v$ is dominating vertex, we have $x \sim v.$ Thus, there exists a cyclic subgroup $A$ such that
	$x, v \in A.$ Then $\text{o}(x) = p$ implies that $x \in \langle v \rangle.$ If $x= v^q$, by division algorithm it can be shown that $q$ has to be a multiple of $r$ and thus $x \in H.$ This completes the proof. 
\end{proof}
We next move on to the most important result of this section.
\begin{theorem}
\label{thm:power_connected_iff_enhcdpower_connected}
For any group $G$, the graph $\mathcal{P}^*(G)$
is connected if and only if the graph $\mathcal{G}_{e}^*(G)$ is connected. 
\end{theorem}
\begin{proof}
The forward implication is easy. That is, if $\mathcal{P}^*(G)$ is connected then $\mathcal{G}_{e}^*(G)$ is of course connected. We prove the other direction. Let, $\mathcal{G}_{e}^*(G)$ be connected and $a,b \in \mathcal{P}^*(G).$ As $\mathcal{G}_{e}^*(G)$ is connected, there exists a path between $a=a_1\sim a_2\sim \cdots\sim a_m=b.$ Now, $a_i \sim a_{i+1}$ in $\mathcal{G}_e^*(G) \implies $ there exist $b_i \in G$ such that both $a_i$ and $ a_{i+1} \in \langle b_i \rangle.$
In that case, $a_i \sim b_i \sim a_{i+1}.$ Therefore, we have 
$a=a_1 \sim b_1 \sim a_2 \sim b_2 \sim a_3 \cdots \sim a_m=b$ in $\mathcal{P}^*(G).$ This completes the proof. 
\end{proof}
Therefore, for any graph $G$, the information about the connectivity of one of the two graphs
$\mathcal{P}^*(G)$ and $\mathcal{G}_e^*(G)$ gives information about the connectivity of the other one. 
\vspace{2 mm} 
\section{Proofs of main results about Vertex Connectivity of $\mathcal{G}_e(G)$ when $G$ is abelian}
\label{sec:3}
\vspace{2 mm}
\begin{proof}[Proof of Theorem \ref{vc=1 for abln and non abln group}]
First suppose that $G$ is non-cyclic abelian $p$-group. Clearly $G$ has at least two distinct cyclic subgroups $H_1= \langle a \rangle$ and $H_2=\langle b \rangle$ of order $p.$ Now by Lemma \ref{lema: p and p^i orderd path connd abelong< b>}, there is no path joining $a$ and $b$ in $\mathcal{G}^*_e(G),$ otherwise $H_1=H_2.$ The proof is complete. 
\end{proof} 
%
\begin{proof}[Proof of Theorem \ref{ enhcd vc 1 iff g is p-group}] 
$G$ is non-cyclic abelian $p$-group. Therefore, by Theorem \ref{vc=1 for abln and non abln group}, $\kappa(\mathcal{G}_e(G))=1.$
	
For the converse part, let $G$ be a finite abelian group which is not a $p$-group. 
Let, $p_1, p_2, \cdots, p_k$ be the prime factors of $|G|.$ 
Let, $a, b \in G$ and $\text{o}(a)= p_1^{r_1} p_2^{r_2} \cdots p_k^{r_k}$
and $\text{o}(b)=p_1^{s_1} p_2^{s_2} \cdots p_k^{s_k}$. We consider the
following two cases: \\
Case 1: There exists distinct $i$ and $j$  with 
$r_i \neq 0$ and $s_j \neq 0$.  
Then the elements ${p_1^{r_1} p_2^{r_2} \cdots p_{i-1}^{r_{i-1}} p_{i+1}^{r_{i+1}} \cdots p_k^{r_k}}a$ and $ {p_1^{s_1} p_2^{s_2} \cdots p_{j-1}^{s_{j-1}} p_{j+1}^{s_{j+1}} \cdots p_j^{s_j}}b$ 
are 
of order $p_i^{r_i}$ and $p_j^{s_j}$ respectively. 
Thus,
by Lemma \ref{lemma:any_odtwo_commuting_elements_of_prime_order_adjacent}, $ {p_1^{r_1} p_2^{r_2} \cdots p_{i-1}^{r_{i-1}} p_{i+1}^{r_{i+1}} \dots p_k^{r_k}} a$ and ${p_1^{s_1} p_2^{s_2} \cdots p_{j-1}^{s_{j-1}} p_{j+1}^{s_{j+1}} \cdots p_j^{s_j}} b$ are adjacent. Therefore we
have 
\[ a \sim {p_1^{r_1} p_2^{r_2} \cdots p_{i-1}^{r_{i-1}} p_{i+1}^{r_{i+1}} \cdots p_k^{r_k}} a \sim  {p_1^{s_1} p_2^{s_2} \cdots p_{j-1}^{s_{j-1}} p_{j+1}^{s_{j+1}} \cdots p_j^{s_j}}
b \sim b.\] That is, there 
exists a path of length $\leq 3$ between $a$ and $b$. 
We observe that this case takes care of everything except 
when both $\text{o}(a)$ and $\text{o}(b)$ are power of the same prime 
$p_{\ell}$ for some $ 1 \leq \ell \leq k$ which we consider next.  \\
Case 2: $\text{o}(a)=p_{\ell}^{r_{\ell}}$ and $\text{o}(b)=p_{\ell}^{s_{\ell}}.$  Let, $c$ 
be an element of order $p_i$ in $G$ with $i \neq \ell$. Then by Lemma \ref{lemma:any_odtwo_commuting_elements_of_prime_order_adjacent},
we have $a \sim c \sim b$. 
Thus, $\mathcal{G}_e^*(G)$ is connected. This completes the proof.
\end{proof}
By Theorem \ref{thm:power_connected_iff_enhcdpower_connected}, we immediately get the following corollary on the connectivity of the power graph. 
 \begin{corollary}
\label{vc 1 iff g is p-group}
Let $G$ be a finite non-cyclic abelian group. Then $\kappa(\mathcal{P}(G))$ is $1$ if and only if $G$ is a $p$-group.  
\end{corollary}
 \begin{proof}[Proof of Theorem \ref{thm: all dom of G times Zn}]
We show that $(e, x)$ is a dominating vertex, where $e$ is the identity element of the group $G$ and $x$ is a any element of the group $\mathbb{Z}_n.$ Consider an arbitrary vertex $(g, y)$  of the graph  $\mathcal{G}_e({G\times \mathbb{Z}_n}).$

\noindent
Case 1: Let $g=e.$ Let $a$ be a generator of the cyclic group $\mathbb{Z}_n.$ Now $y\in \mathbb{Z}_n$ implies that $(e, y), (e, x)\in \langle(e, a)\rangle$ and so $(e, x)\sim (e, y).$
	
\noindent
Case 2: Let $g\neq e$ and $y=0.$ [Here $0$ actually means the additive identity of the group $\mathbb{Z}_n].$  We show that $(g, 0), (e, x)\in \langle(g, a)\rangle.$ First we show that $(e, a)\in \langle(g, a)\rangle.$ Let $\text{o}(g)=m.$ Now $\text{gcd}(|G|, n)=1$ implies that $\text{gcd}(m, n)=1.$ Then by Euler's Theorem $m^{\phi(n)}=n\ell+1, \ell\in \mathbb{N}.$ Therefore, $(g, a)^{m^{\phi(n)}}=(g^{m^{\phi(n)}}, a^{m^{\phi(n)}})=(e, a^{n\ell+1})=(e, a).$ Hence, $(e, a)\in \langle (g, a)\rangle.$ Now we show that $(g, 0)\in \langle(g, a)\rangle.$  It is given that $\text{gcd}(m, n)=1.$ So, by the Euler's theorem, $n^{\phi(m)}=mk+1, k\in \mathbb{N}.$ Hence $(g, a)^{n^{\phi(m)}}=(g^{n^{\phi(m)}}, 0)=(g^{mk+1}, 0)=(g, 0).$ Consequently, $(g, 0)\in \langle (g, a) \rangle.$
	
\noindent
Case 3: Let $g\neq e $ and $y\neq 0.$  We show that $(g, y), (e, x)\in \langle(g, a)\rangle.$  Already we have proved  that $(g, 0), (e, x)\in \langle(g, a)\rangle.$  Since $a$ is a generator of $\mathbb{Z}_n$, $(e, y)\in \langle(e, a)\rangle \subset \langle(g, a)\rangle.$ Hence $ (g, y)=(g, 0)(e, y)\in \langle(g, a)\rangle.$ 

To finish the proof we have to show that if $(g, z)$ is a dominating vertex, then $g$ must be the identity of $G.$ Let \[G=\mathbb{Z}_{p^{t_{11}}_1}\times \mathbb{Z}_{p^{t_{12}}_1}\times\mathbb{Z}_{p^{t_{13}}_1}\times\cdots\times\mathbb{Z}_{p^{t_{1k_1}}_1}\times \mathbb{Z}_{p^{t_{21}}_2}\times\mathbb{Z}_{p^{t_{22}}_2}\times\cdots\times\mathbb{Z}_{p^{t_{2k_2}}_2}\times\cdots\times
\mathbb{Z}_{p^{t_{r1}}_r}\times\mathbb{Z}_{p^{t_{r2}}_r}\times\cdots\times\mathbb{Z}_{p^{t_{rk_r}}_r},\] where $k_i\geq2 $ and $t_{i1}\leq t_{i2}\leq\cdots\leq t_{ik_i} $, for all $i\in [r]$. Let \[v=(x_{p^{t_{11}}_1}, x_{p^{t_{12}}_1}, \cdots x_{p^{t_{1k_1}}_1}, x_{p^{t_{21}}_2}, \cdots, x_{p^{t_{2k_2}}_2}, \cdots, x_{p^{t_{1r}}_r}, x_{p^{t_{2r}}_r}, \cdots, x_{p^{t_{rk_r}}_r}, z)\] be a dominating vertex. We will  prove that, for each $i\in [r]$ and $j\in [k_i], x_{p^{t_{i1}}_i}=x_{p^{t_{i2}}_i}=\cdots=x_{p^{t_{ij}}_i}=0.$ [Here $0$ actually means the additive identity of the group $\mathbb{Z}_{p^{t_{ij}}_i}$]. Consider the element \[v'=(0, 0, \cdots, 0, g_{p^{t_{1k_1}}_1}, 0, 0, \cdots, g_{p^{t_{2k_2}}_2}, 0, \cdots, 0, g_{p^{t_{rk_r}}_r}, z'),\] where  $g_{p^{t_{ik_i}}_i}$ is a generator of the cyclic group $\mathbb{Z}_{p^{t_{ik_i}}_i}$ for each $i\in[r]$ and $z'$ is a generator of $\mathbb{Z}_n.$ As $v$ is a dominating vertex of the graph $\mathcal{G}_e(G),$ we have $v\sim v'$ in $\mathcal{G}_e(G).$ Clearly, $v'$ is an element of maximum ordered. So, we have $v\in\langle v'\rangle.$  As a result, for each $i\in [r]$ and $j\in [k_i-1], x_{p^{t_{i1}}_i}=x_{p^{t_{i2}}_i}=\cdots=x_{p^{t_{ij}}_i}=0,$ i.e., \[v=(0, 0, \cdots, 0, x_{p^{t_{1k_1}}_1}, 0, 0, \cdots, 0, x_{p^{t_{2k_2}}_2}, 0, \cdots, 0, x_{p^{t_{rk_r}}_r}, z).\] Now we show that $x_{p^{t_{ik_i}}_i}=0,$ for all $i\in[r].$ Suppose at least one of the $x_{p^{t_{ik_i}}_i}$ is non-zero. Without any loss of generality we assume that $x_{p^{t_{1k_1}}_1}\neq 0.$ Consider $v_1=(x, 0, 0, \cdots, 0, 0, 0),$ ({last zero is the identity of cyclic group $Z_n)$} where $x\in \mathbb{Z}_{p^{t_{1k_1}}_1}$ with $\text{o}(x)=p_1.$ Then $p_1 \text{ divides }\text{o}(v).$ If $v\sim v_1,$ then there exists a cyclic subgroup $C$ of $G\times \mathbb{Z}_n$ such that $v, v_1\in C.$ Then $\text{o}(v_1)=p_1$ and $p_1$ divides $\text{o}(v)$ implies that $v_1\in\langle v\rangle,$ which contradicts that $x\neq 0.$ This completes the proof.
\end{proof}
We next take care about the connectivity of the proper enhanced power graph $\mathcal{G}^{**}_e(G),$ when $G$ is abelian. 
\begin{proof}[Proof of Theorem \ref{thm:**disconnected}]
First we show that if $G_1$ is a $p$-group, then $\mathcal{G}^{**}_e(G)$ is disconnected. By Theorem \ref{thm: all dom of G times Zn}, order of each element of $\mathcal{G}_e^{**}(G)$ is divisible by $p.$ So applying the proof of the lemma \ref{lema: p and p^i orderd path connd abelong< b>}, we get that for two elements $a= (x, 0)$ and $b=(x', y')$ of $\mathcal{G}_e^{**}(G),$ with $\text{o}((x, 0))=p$ and $x'\neq e,$ if there exists any path joining $a$ and $b,$ then $ \langle (x, 0)  \rangle $ is contained in or equal to $ \langle (x', y') \rangle .$ In particular, if both $a$ and $b$ have order $p,$ then the existence of a path joining $a$ and $b$ implies that $\langle a \rangle = \langle b \rangle .$ Since, $G_1$ is noncyclic abelian $p$-group, there exist two elements $a$ and $b$ of order $p$ such that $\langle a \rangle \neq \langle b \rangle .$ So by our previous observation, $a$ is not path connected to $b.$

Conversely, suppose that $G_1$ is non-$p$-group. Then we show that $\mathcal{G}^{**}(G)$ is connected. Here we have two cases. 

Case 1: Let $(x_1, 0)$ and $(x_2, 0)$ be two elements of $G_1\times\mathbb{Z}_n$ such that $x_1\neq e$ and $x_2\neq e.$ Then by same argument as in proof of converse part of Theorem \ref{ enhcd vc 1 iff g is p-group}, $(x_1, 0)$ and $(x_2, 0)$ are path connected in $\mathcal{G}^{**}_e(G).$

Case 2: Let $(x_1, y_1), (x_2, y_2)\in V(\mathcal{G}^{**}_e(G)).$ Clearly, $(x_1, y_1)\sim (x_1, 0)$ and $(x_2, 0)\sim (x_2, y_2).$ Again there is a path between $(x_1, 0)$ and $(x_2, 0)$ in $\mathcal{G}^{**}_e(G)$ by Case 1. Therefore, $(x_1, y_1)$ and $(x_2, y_2)$ are path connected in $\mathcal{G}^{**}_e(G).$ Hence the graph $\mathcal{G}^{**}_e(G)$ is connected. This completes the proof.
\end{proof}
\begin{proof}[Proof of Theorem \ref{vc of enhced pwr raph for grnrral abelian grp}]
	Let $H=\langle a\rangle$ be a cyclic subgroup of $G,$ where $a=(a_{11}, 0, \cdots, 0, a_{21},0,  \cdots, a_{r1}, 0\cdots, 0),$ and $a_{i1}\in \mathbb{Z}_{p_i^{t_{i1}}}$ such that o$(a_{i1})=p_i^{t_{i1}},$ for $i=1, 2, \cdots, r.$
	$H$ is maximal cyclic subgroup of $G.$ 
	Now we show that for any $b\in G\setminus H,$ there is no edge between $b$ and any element in $\mathcal{G}en(a).$ If possible there exists $x\in \mathcal{G}en(a)$ such that $b\sim x$ in $\mathcal{G}_e(G).$ Then there exists a cyclic subgroup $K$ of $G$ such that $b, x\in K.$ Again $H$ is a maximal cyclic subgroup of $G$ which is also generated by $x$. Therefore, $K=\langle a\rangle=H.$ Hence a contradiction as $b\in G\setminus H.$ Clearly, if we remove the identity and non-identity non-generators elements from the cyclic subgroup $H,$ then the graph will be disconnected and the number of deleted vertices is $p_1^{t_{11}}p_2^{t_{21}}\cdots p_r^{t_{r1}}-\phi(p_1^{t_{11}}p_2^{t_{21}}\cdots p_r^{t_{r1}}).$ Hence the result.
\end{proof}
From Theorem \ref{vc of enhced pwr raph for grnrral abelian grp}, we immediately have the following corollary on the vertex connectivity of power graphs of any non-cyclic abelian group.  
\begin{corollary}
\label{cor:vc_abelian_group_power_graph}
Let $G$ be a non-cyclic abelian group such that \[G\cong\mathbb{Z}_{p^{t_{11}}_1}\times \mathbb{Z}_{p^{t_{12}}_1}\times\mathbb{Z}_{p^{t_{13}}_1}\times\cdots\times\mathbb{Z}_{p^{t_{1k_1}}_1}\times \mathbb{Z}_{p^{t_{21}}_2}\times\mathbb{Z}_{p^{t_{22}}_2}\times\cdots\times\mathbb{Z}_{p^{t_{2k_2}}_2}\times\cdots\times
\mathbb{Z}_{p^{t_{r1}}_r}\times\mathbb{Z}_{p^{t_{r2}}_r}\times\cdots\times\mathbb{Z}_{p^{t_{rk_r}}_r},\] where $k_i\geq 1 $ and $1\leq t_{i1}\leq t_{i2}\leq\cdots\leq t_{ik_i} $, for all $i\in [r]. $ Then \[\kappa(\mathcal{P}(G))\leq p_1^{t_{11}}p_2^{t_{21}}\cdots p_r^{t_{r1}}-\phi(p_1^{t_{11}}p_2^{t_{21}}\cdots p_r^{t_{r1}}).\]
\end{corollary}
\begin{proof}[Proof of Theorem \ref{vc of G1 cross cyclic}]
Proof of this theorem follows from Theorems \ref{thm: all dom of G times Zn} and \ref{thm:**disconnected}. 
\end{proof}
\subsection{Number of Components of $\mathcal{G}_{e}^{**}(G)$ when it is disconnected} 
So far, we have characterized the abelian groups for which the proper enhanced power graph is disconnected. In this context, the 
natural question that comes to our mind is the number
of connected components of the subgraph 
$\mathcal{G}_e^{**}(G).$  By Theorem \ref{thm:**disconnected}, the proper enhanced power graph $\mathcal{G}_e^{**}(G)$ for a finite abelian group $G$ is disconnected when
$G$ is either a non-cyclic $p$-group or $G \cong G_1 \times \mathbb{Z}_n$ where $G_1$ is an $p$-group and $\text{gcd}(p, n)=1.$ Here, we explicitly count the number of components for those $G$. 
\begin{theorem}
	\label{thm:components_of_G_abelian}
	Let G be a finite abelian $p$-group. Suppose
	$$G = Z_{p^{t_1}}
	\times  Z_{p^{t_2}}
	\times  \cdots \times 
	Z_{p^{t_r}}.$$ where $r \geq 2$ and $t_1 \leq t_2 \leq \cdots \leq t_r.$
	Then, the number of components of $\mathcal{G}_{e}^{**} (G )$  
	is $\frac{p^r-1}{p-1}.$ 
\end{theorem}
\begin{proof}
	It is easy to show that there are $p^{r}-1$ elements of order $p$.
	For any element $a$ of order $p$,  the $p-1$ non-zero 
	scalar multiples of $a$ must be in the same component.
	Moreover, by Lemma \ref{lema: p and p^i orderd path connd abelong< b>}, 
	if any two elements of order $p$ are connected by a path, 
	then one of them must be the multiple of another. Henceforth, 
	there are exactly  $p-1$ members of order $p$ in any component. 
	Thus, the number of connected components of 
	$\mathcal{G}_{e}^* (G )$ is $\frac{p^{r}-1}{p-1}.$ 
\end{proof}
It is quite interesting to note that the number of components of the proper enhanced power graph of a finite abelian non-cyclic $p$-group is independent of the exponent $t_i$'s. In the next result, we prove that this phenomenon is observed also in the case when $G \cong G_1 \times \mathbb{Z}_n$ where $G_1$ is an $p$-group and $\text{gcd}(p, n)=1.$
Let $C(\mathcal{G}_{e}^{**} (G ))$ be the set of connected components of the proper enhanced power graph $\mathcal{G}_{e}^{**} (G).$
\begin{theorem}
	\label{thm:no_of_components}
	Let $G$ be an abelian group such that  $G \cong \mathbb{Z}_{p^{t_{1}}}\times \mathbb{Z}_{p^{t_{2}}}\times \cdots \times \mathbb{Z}_{p^{t_{r}}} \times \mathbb{Z}_n,$ where $r \geq 2$ and $\text{gcd}(p, n)=1.$ Then, the number of components of  $\mathcal{G}_{e}^{**} (G )$ is $\frac{p^r-1}{p-1}.$ 
\end{theorem}
\begin{proof}
	Let, $G_1= \mathbb{Z}_{p^{t_{1}}}\times \mathbb{Z}_{p^{t_{2}}}\times \cdots \times \mathbb{Z}_{p^{t_{r}}}.$ By Theorem \ref{thm:components_of_G_abelian}, the number of connected components of  $\mathcal{G}_{e}^{**} (G_1)$ is $\frac{p^r-1}{p-1}.$ Let, $C_1, C_2, \cdots, C_{\frac{p^r-1}{p-1}}$ be the components of 
	$\mathcal{G}_{e}^{**} (G_1) \mapsto C(\mathcal{G}_{e}^{**} (G_1 \times \mathbb{Z}_n))$ by
	$$f(C_i)= C_i \times  \mathbb{Z}_n .$$
	
	At first, we show that there is no path in between $C_i \times  \mathbb{Z}_n$ and 
	$C_j \times  \mathbb{Z}_n$ for $1 \leq i < j \leq \frac{p^r-1}{p-1}.$ Let, there exists 
	an path between $(a_1, b_1)$ and $(a_2, b_2)$ where $a_1 \in C_i$, $a_2 \in C_j$ and $b_1, b_2 \in \mathbb{Z}_n.$ 
	If possible, let $(a_1, b_1) \sim (c_1, d_1) \sim (c_2, d_2) \sim \cdots \sim (c_{m-1}, d_{m-1}) \sim (a_2, b_2)$ in $\mathcal{G}_e^{**}(G)$ where $c_1, c_2,\dots, c_{m-1} \in G_1$ and $d_1, d_2, \dots,d_{m-1} \in \mathbb{Z}_n$. Then $c_1, c_2, \cdots, c_{m-1}$ must be non-zero elements of $G_1$. This proves that $a_1$ and $a_2$ are connected by a path in $\mathcal{G}_e^{*}(G_1)$ which contradicts the fact that $C_i$ and $C_j$ are distinct connected components of $\mathcal{G}_{e}^{**} (G_1).$ Therefore, the number of components of $\mathcal{G}_{e}^{**} (G_1)$ is at least 
	$\frac{p^r-1}{p-1}.$ 
	
	Moreover, it is clear that the number of elements of order $p$ in $G$ is $\frac{p^r-1}{p-1}.$ Any element of order $>p$ is adjacent to an element of order $p$. Therefore, the number of components of $\mathcal{G}_{e}^{**} (G_1)$ should be exactly equal to 
	$\frac{p^r-1}{p-1}.$ The proof is complete. 
\end{proof}
\section{Vertex Connectivity of Some Non-abelian groups}
\label{sec:5}
In this Section, we discuss the vertex connectivity of some interesting classes of non-abelian groups. We start with the dihedral groups. We need the structures of these groups to determine the vertex connectivity. 
For $n \geq 2$, the dihedral group of order $2n$ is defined by the following presentation:
\[ D_{2n}= \langle r, s : r^n=s^2=e, rs=sr^{-1} \rangle.\]
We also consider the generalized quarternion groups $Q_{2^n}.$ Let $x = \overline{(1, 0)}$ and $y = \overline{(0, 1)}.$ Then $Q_{2^n} = \langle x, y\rangle,$ where
\begin{enumerate}
\item
$x$ has order $2^{n-1}$ and $y$ has order $4,$
\item
every element of $Q_{2^n}$ can be written in the form $x^a$ or $x^ay$ for some $a\in \mathbb{Z},$
\item
$x^{2^{n-2}}=y^2,$
\item
for each $g\in Q^{2^n}$ such that $g\in \langle x \rangle,$ such that $gxg^{-1}=x^{-1}.$
\end{enumerate}
For more information about $D_{2n},$ and $Q_{2^n}$ see \cite{generalized-quaternion, algebradummitfoote, scott-group}. 
\begin{theorem}
\label{thm:components_of_G_dihedral}
Let $G$ be the dihedral group of order $2n.$
Then $\kappa(\mathcal{G}_e(G))$ is $1.$
Moreover, the number of components of $\mathcal{G}_e^{**}(G)$  is
$n+1$. 
\end{theorem}
%
\begin{proof}
Consider the following $n+1$ sets:
\begin{eqnarray}
\label{eqn:constructionofVi}
S_1 = \{ rs  \}, S_2= \{ r^2s \}, \cdots, S_{n-1}= \{ r^{n-1}s \}, S_n  = \{  s  \},  S_{n+1}=\{ r,r^2,\dots,r^{n-1}\}.\nonumber 
\end{eqnarray}
We observe that $G\setminus \{e\} = \displaystyle \cup_{i=1}^{n+1} S_i $ 
and for  $1 \leq i <  j \leq n$. Moreover, the power of any element of $S_i$ must be in $S_i$ itself in $G\setminus \{e\}.$  
Therefore, there can be no edge between $S_i$ and $S_j$ for distinct $i,j$. 
This completes the proof. 
\end{proof}
\begin{theorem}\label{connectivty generalized qtrn}
For $n\geq 3,$ let $Q_{2^n}$ be the generalized quaternion group. Then the vertex connectivity of $\mathcal{G}_e(Q_{2^n})$ is $2.$ Moreover the number of components of $\mathcal{G}^{**}_e(Q_{2^n})$ is $2^{n-2}+1.$ 	
\end{theorem}
\begin{proof}
$Q_{2^n}$ is generalized quaternion group, so $Q_{2^n}$ is a $2$-group and it has a unique minimal subgroup of order $2.$ Let $g\in Q_{2^n}$ such that $\text{o}(g)=2.$ Then by Lemma \ref{dom of gen Q_2^n in enhced}, $g$ and $e$ are adjacent to all other vertices in $\mathcal{G}_e(Q_{2^n}).$ For this reason to disconnect the graph we have to delete the vertices $e, g.$ Consequently, $\kappa(\mathcal{G}_e(Q_{2^n}))\geq 2.$
Now we show that after removing the vertices $e$ and $g$ from $V(\mathcal{G}_e(Q_{2^n})),$ the graph $\mathcal{G}^{**}_e(Q_{2^n})$ will be disconnected. Let $x\in Q_{2^n}$ such that $\text{o}(x)=2^{n-1}$ and $H=\langle x\rangle.$ We will prove that there is no edge between the vertices in $H$ and $Q_{2^n}\setminus H.$ If for any $y\in Q_{2^n}\setminus H$ is adjacent to a vertex of $H,$ then $y$ should belong to $H,$ (since $H$ is the only subgroup of order $2^{n-1}$ and there is no other subgroup of order $>2^{n-1}$) a contradiction. So the graph is disconnected. Moreover it is clear that there is no edge between the vertices in $\mathcal{G}en(y)$ and $\mathcal{G}en(z),$ where $y, z\in Q_{2^n}\setminus H$ such that $\text{o}(y)=4=\text{o}(z)$ and $\langle y\rangle\neq\langle z\rangle.$ Hence the number of components in $\mathcal{G}^{**}_e(Q_{2^n})$ is $2^{n-2}+1$ (as the number of $4$-ordered element in $Q_{2^n}\setminus H$ is $2^{n-1} \text{ and } \phi(4)=2).$ 
\end{proof}
\begin{corollary}\label{connectivty **generalized qtrn of enhed power graph}
Let $Q_{2^n}$ be the generalized quaternion group. Then the enhanced power graph $\mathcal{G}^{*}_e(Q_{2^n})$ is connected but the proper enhanced power graph $\mathcal{G}^{**}_e(Q_{2^n})$ is disconnected.
\end{corollary}
We next consider the family of symmetric groups $\mathfrak{S}_n.$ \cite{algebradummitfoote} is a good reference for this. Recall from Section \ref{sec:intro} that $\mathcal{P}(G)$ denotes
the power graph of $G.$ Let,  $\mathcal{P}^*(G)$ denote
the power graph of $G$ after deleting the identity. Doostabadi et.al proved the following theorem on the vertex connectivity of power graphs \cite{doostabadi}.
\begin{theorem}
\label{thm:doostabadi_power_graph}
Let $G = \mathfrak{S}_n$ be a symmetric group with $n \geq 3.$ Then 
\begin{enumerate}
\item If $n \geq 3$ and neither $n$ nor $n-1$ is a prime, then $\mathcal{P}^*(G)$ is connected.
\item If $n$ is such that either $n$ or $n-1$ is a prime, then $\mathcal{P}^*(G)$ is disconnected.
\end{enumerate}
\end{theorem}
We prove an analogous result corresponding to the enhanced power graph of $\mathfrak{S}_n.$
For this, we first prove that $\mathcal{G}_e (\mathfrak{S}_n)$ has no dominating vertex. 
\begin{lemma}
\label{lem:Symmetricgrouphasnodominatingvertex}
For $n \geq 3, $ the enhanced power graph 
$\mathcal{G}_e(\mathfrak{S}_n)$ has no dominating vertex other than identity. Therefore, for $n \geq 3$, the graphs  $\mathcal{G}_e^{*}(\mathfrak{S}_n)$ and
$\mathcal{G}_e^{**}(\mathfrak{S}_n)$ coincide.
\end{lemma}
\begin{proof}
We first prove it when $n$ is composite. Let $a \neq e$ be a dominating vertex. By Lemma \ref{lem:structurelemmauniqueprimeordersubgroup}, there exists a prime $p$ dividing $\text{o}(a) $ such that $G$ has a unique subgroup of order $p.$ But we can take $C_1= \langle (1, 2, \cdots, p) \rangle$ and $C_2=\langle (2, 3, \cdots, p+1) \rangle$ and arrive at a contradiction. We next consider the case when $n$ is prime, say $n=p.$ Let $a \neq e$ be a dominating vertex. Then, $a \sim \langle(1, 2, \cdots, p)\rangle.$ So, they are contained in a cyclic subgroup, say $A$ of $G.$ Now, there cannot be a subgroup of $G$ which properly contains
$\langle(1, 2, \cdots, p)\rangle.$ Hence, $a \in \langle(1, 2, \cdots, p)\rangle$ and consequently $ \langle a \rangle = \langle(1, 2, \cdots, p)\rangle.$
Again, since $a$ is a  dominating vertex, $a \sim (1,2).$ Now, applying the similar argument as above, we see that $(1,2) \in \langle a \rangle,$ which is not possible. The proof is complete.
\end{proof}
\begin{theorem}
\label{thm:vertex-connectivity_of_sym_groups}
For positive integers $n \geq 3$, 
$\kappa (\mathcal{G}_e(\mathfrak{S}_n)) = 1$ if and 
only if either $n$ or $n-1$ is prime. 
\end{theorem}
\begin{proof}
By Lemma \ref{lem:Symmetricgrouphasnodominatingvertex}, we have 
$\mathcal{G}_e^{*}(\mathfrak{S}_n)=\mathcal{G}_e^{**}(\mathfrak{S}_n).$ So, by Theorem \ref{thm:power_connected_iff_enhcdpower_connected}, $\mathcal{G}_e^{**}(\mathfrak{S}_n)$ is connected if and only if $\mathcal{P}^{*}(\mathfrak{S}_n)$ is connected. Now, the proof is complete using Theorem \ref{thm:doostabadi_power_graph}. 
\end{proof}
Let, $\mathcal{A}_n \subseteq \mathfrak{S}_n$ be the alternating group. The family of Alternating groups is an interesting subgroup of the set of even permutations in $\mathfrak{S}_n.$ Doostabadi et.al proved the following theorem on the vertex connectivity of power graphs \cite{doostabadi}.
\begin{theorem}
\label{thm:doostabadi_power_graph_alternating}
Let $G = \mathcal{A}_n$ be the alternating group and $n \geq 4$. Then 
\begin{enumerate}
\item If $n, n-1, n-2, n/2, (n-1)/2, (n-2)/2 $ are not primes, then $\mathcal{P}^*(G)$ is connected.
\item If $n$ is such that any one of $n, n-1, n-2, n/2, (n-1)/2, (n-2)/2 $ is prime, then $\mathcal{P}^*(G)$
is not connected.
\end{enumerate}
\end{theorem}
We start with showing that $\mathcal{G}_e (\mathcal{A}_n)$ has no dominating vertex.
\begin{lemma}
\label{lem:alternatinggrouphasnodominatingvertex}
For $n \geq 4, $ the enhanced power graph 
$\mathcal{G}_e(\mathcal{A}_n)$ has no dominating vertex other than identity. 
Therefore, for $n \geq 4,$ the graphs  $\mathcal{G}_e^{*}(\mathcal{A}_n)$ and
$\mathcal{G}_e^{**}(\mathcal{A}_n)$ coincide.
\end{lemma}
\begin{proof}
Let, $a\neq e$ be a dominating vertex.  By Lemma \ref{lem:structurelemmauniqueprimeordersubgroup}, there exists a prime $p$ such that $G$ has a unique subgroup of order $p.$ If $p \geq 4$, take  $C_1= \langle (1, 2, 3, \cdots, p) \rangle$ and   $C_2= \langle (1, 3, 2, \cdots, p) \rangle.$ If  $p=3,$ take $C_1= \langle (1, 2, 3) \rangle$ and   $C_2= \langle (2, 3, 4) \rangle.$ When  $p=2,$ let $C_1= \langle (1, 2)(3, 4) \rangle$ and   $C_2= \langle (1, 3)(2, 4) \rangle.$ Thus, in each case, we can verify that both $C_1$ and $C_2$ are even permutations and this contradicts Lemma \ref{lem:structurelemmauniqueprimeordersubgroup}. This completes the proof. 
\end{proof}
From Theorem \ref{thm:doostabadi_power_graph_alternating} and Lemma \ref{lem:alternatinggrouphasnodominatingvertex}, we prove the following theorem. 
\begin{theorem}
\label{thm:vertex-connectivity_of_alt_groups}
For positive integers $n \geq 7$, 
$\kappa (\mathcal{G}_e(\mathcal{A}_n)) = 1$ if and 
only if one of $n, n-1, n-2, n/2, (n-1)/2, (n-2)/2$ 
is prime. 
\end{theorem}
\begin{proof}
By Lemma \ref{lem:alternatinggrouphasnodominatingvertex}, 
we have 
$\mathcal{G}_e^{*}(\mathfrak{S}_n)=\mathcal{G}_e^{**}(\mathfrak{S}_n).$ Now, we are done using Theorem \ref{thm:power_connected_iff_enhcdpower_connected} and Theorem \ref{thm:doostabadi_power_graph_alternating}.
\end{proof}
\subsection*{Acknowledgement}
The First author would like to thank Prof. Arvind Ayyer for his constant support and encouragement. The second author would like to thank Prof. Sivaramakrishnan Sivasubramanian for his constant support and encouragement. The third author would like to thank Prof. Basudeb Datta for his constant support and encouragement. The first author was supported by Department of Science and Technology grant EMR/2016/006624 and partly supported by  UGC Centre for Advanced Studies. Also the first author was supported by NBHM Post Doctoral Fellowship grant 0204/52/2019/RD-II/339. The second author was supported by a CSIR-SPM fellowship. The third author was supported by NBHM Post Doctoral Fellowship grant 0204/3/2020/RD-II/2470. 
\bibliographystyle{amsplain}
\bibliography{gen-inv-lcp.bib}
\end{document}